\documentclass[11pt]{article}

\usepackage[margin=1.1in]{geometry}
\usepackage{amsmath,amssymb,amsthm,mathtools}
\usepackage{bm}
\usepackage{graphicx}
\usepackage{booktabs}
\usepackage{enumitem}
\usepackage{xcolor}
\usepackage[colorlinks=true,linkcolor=blue!60!black,citecolor=blue!60!black,urlcolor=blue!60!black]{hyperref}
\usepackage[capitalise,noabbrev]{cleveref}

\newtheorem{theorem}{Theorem}[section]
\newtheorem{proposition}[theorem]{Proposition}
\newtheorem{corollary}[theorem]{Corollary}
\newtheorem{lemma}[theorem]{Lemma}
\theoremstyle{definition}
\newtheorem{definition}[theorem]{Definition}
\newtheorem{example}[theorem]{Example}

\theoremstyle{remark}
\newtheorem{remark}[theorem]{Remark}

\newcommand{\R}{\mathbb{R}}
\newcommand{\dd}{\mathrm{d}}
\newcommand{\bz}{\bm{z}}
\newcommand{\bx}{\bm{x}}
\newcommand{\bE}{\bm{E}}
\newcommand{\bH}{\bm{H}}
\newcommand{\bK}{\bm{K}}
\newcommand{\bL}{\bm{L}}
\newcommand{\bJ}{\bm{J}}
\newcommand{\bP}{\bm{P}}
\newcommand{\bM}{\bm{M}}
\newcommand{\bI}{\bm{I}}
\newcommand{\chixx}{\chi(\bx)}
\newcommand{\Lie}{\mathcal{L}}
\newcommand{\Do}{\mathcal{D}}
\newcommand{\So}{\mathcal{S}}
\newcommand{\Go}{\mathcal{G}}
\DeclareMathOperator{\diag}{diag}
\DeclareMathOperator{\ran}{ran}
\DeclareMathOperator{\curl}{curl}

\begin{document}

\title{\bfseries Elucidating the Conformal Structure of the Brinkman Penalisation Method for Geometry-Adapted, Structure-Preserving Operator Learning of Hamiltonian PDEs}
\author{Teo Deveney\thanks{Department of Computer Science, University of
Bath, United Kingdom.}%
\and
Baige Xu\thanks{Center for Advanced Intelligence Project, RIKEN,
Japan.}%
\and
Takaharu Yaguchi\thanks{Kyushu University / Center for Advanced
Intelligence Project, RIKEN, Japan
(\texttt{yaguchi@math.kyushu-u.ac.jp}).}}

\date{}
\maketitle

\begin{abstract}
The Brinkman penalisation method embeds boundary-value problems on complex domains into a simple computational box by modeling the solid region as a strongly dissipative medium, avoiding body-fitted mesh generation. We show that multi-symplectic Hamiltonian PDEs regularised by Brinkman-type penalisation retain a multi-conformal symplectic structure under a compatibility condition linking the symplectic matrix and the penalisation projection. This yields an exact local conservation law, under which the multi-symplectic two-form is conserved in the fluid region and decays exponentially inside the solid. The linear wave equation with Brinkman friction and Maxwell’s equations with artificial Ohmic conductivity satisfy this condition, with explicit modified Hamiltonian densities. Building on this, we propose (i) structure-preserving numerical integrators via Strang splitting that satisfy a discrete conformal conservation law, and (ii) conformal symplectic neural operators that interleave exact dissipative flows with learnable multi-symplectic evolution operators, allowing geometry-dependent operator learning. Numerical experiments on wave and electromagnetic scattering demonstrate that our methods reproduce correct local energy budgets and avoid unphysical energy drift, providing a principled framework for physics-consistent scientific machine learning on complex domains.
\end{abstract}



\section{Introduction}
\label{sec:intro}

Scientific machine learning (SciML), the fusion of scientific
computing and machine learning, has produced a rapidly growing family
of methods for the simulation and identification of physical systems,
including physics-informed neural networks
\cite{raissi2019physics}, neural operators such as DeepONet
\cite{lu2021learning} and the Fourier neural operator (FNO)
\cite{li2021fourier}, and data-driven system identification techniques
based on the Koopman operator, dynamic mode decomposition, sparse
regression, and Hamiltonian neural networks
\cite{schmid2010dynamic,brunton2016discovering,greydanus2019hamiltonian}.
A principal promise of these methods is the acceleration of physical
simulation and, consequently, of design and product-development
workflows in engineering.

Two obstacles, however, persistently limit the deployment of
operator-learning methods on realistic problems.

\paragraph{Complex domains.}
Solution operators of boundary-value problems depend on the domain
$\Omega$ on which the equations are posed, however learning such operators so that they generalize over families of domains is difficult. Classical discretisations based on body-fitted meshes require costly mesh generation and remeshing whenever the shape changes, and neural surrogates trained on a fixed geometry typically require retraining for each new shape.  The \emph{Brinkman penalisation method} \cite{angot1999penalisation,kolomenskiy2009fourier} circumvents
body-fitted meshing altogether by modelling the solid region $\Omega_s$
as a dissipative medium with boundary condition imposed
approximately by a penalty term supported on $\Omega_s$. This allows the
problem to be solved on a simple Cartesian grid regardless of geometry by choosing an appropriate indicator function penalty. 
\begin{equation}
\label{eq:indicator}
\chixx \;=\;
\begin{cases}
1, & \bx\in\Omega_s \quad\text{(solid domain)},\\[2pt]
0, & \bx\in\Omega_f \quad\text{(fluid domain)},
\end{cases}
\end{equation}
which, crucially for operator learning, can itself be supplied as an
input field to a neural operator.

\paragraph{Data quality.}
The training data used in SciML are frequently generated by
conventional numerical methods that do not respect the physical
invariants of the underlying equations (e.g explicit Euler timestepping typically exhibit systematic energy drift).
Errors of this kind are often present within training data, and models
trained on this physically inaccurate data inherit and often amplify these defects, leading to failure in long-term prediction.  A more subtle issue is that models that do not respect physical structure may fit such flawed data better than structure-preserving models do, meaning naive empirical risk minimisation actively leads to emulating physically incorrect behaviour.
Structure-preserving (geometric) numerical integration
\cite{hairer2006geometric} addresses this issue at the source by
preserving invariants and geometric structures at the discrete level.
This guarantees stability and physical fidelity in long-time
integration, and thereby yields high-quality training data.  The same
philosophy has motivated structure-preserving network architectures
such as SympNets \cite{jin2020sympnets}, which employ symplectic maps
as layers.

\paragraph{Contributions.}
This paper addresses these challenges for the case of general
Hamiltonian PDEs. Our starting point is a demonstration that Brinkman-type penalisations of multi-symplectic Hamiltonian PDEs are more than a convenient penalisation method and, in fact, possesses an exact geometric structure of its own. This observation then enables the subsequent study and algorithmic design through the lens of appropriate structure preserving methods that we present in this work. More specifically, our contributions are as follows.
\begin{enumerate}[leftmargin=2em]
\item \emph{Geometry of penalised Hamiltonian PDEs
(\cref{sec:structure}).}
We consider multi-symplectic PDEs
$\bK\bz_t+\sum_k\bL_k\bz_{x_k}=\nabla S(\bz)$ in the sense of Bridges
and Reich \cite{bridges1997multi,bridges2001multi}, penalised by a
friction term $-(\chixx/\eta)\,\bP\bK\bz$ that damps the subset of
evolution equations selected by an orthogonal projection $\bP$.  Our
main result (\cref{thm:general}) states that whenever the
\emph{compatibility condition}
\begin{equation*}
\bK \;=\; \bP\bK + \bK\bP
\end{equation*}
holds, the penalised system is
multi-conformal symplectic in the sense of
\cite{mclachlan2001conformal,moore2013conformal}, with the explicitly
modified Hamiltonian density
$S_\chi(\bz)=S(\bz)+\frac{\chixx}{4\eta}\,
\bz^{\mathsf T}(\bI-2\bP)\bK\bz$, and its solutions obey
\begin{equation*}
\partial_t\omega + \sum_{k}\partial_{x_k}\kappa_k
= -\frac{\chixx}{\eta}\,\omega.
\end{equation*}
That is, the symplectic density is conserved in the fluid region and
decays at the exponential rate $1/\eta$ in the solid region.  The
linear wave equation with the classical Brinkman friction,
and Maxwell's equations penalised by an
artificial Ohmic conductivity both satisfy the
compatibility condition. 
\item \emph{Structure-preserving integrators
(\cref{sec:integrators}).}
Motivated by the structure identified above, we construct numerical
methods for the whole penalised class by Strang splitting. In our scheme the dissipative part of the flow is a pointwise linear ODE that can be integrated exactly, and the conservative part is treated by any multi-symplectic integrator.  A key result
(\cref{lem:dissipative-contraction}) shows that, under the
compatibility condition, the exact geometry-aware dissipative map contracts the symplectic density pointwise by exactly $e^{-\chi\tau/\eta}$.
Consequently, the composed schemes satisfy a discrete conformal
multi-symplectic conservation law (\cref{prop:discrete-law}) and
reproduce the correct local energy budget of the penalised equations.
\item \emph{Conformal symplectic neural operators
(\cref{sec:operators}).}
Lifting the SympNet approach \cite{jin2020sympnets} from finite-dimensional phase spaces to
function spaces, we propose neural operators of the form
$\Go_\theta = \Do^{\chi,\eta}_{\Delta t/2}\circ\So_{\theta,\Delta
t}\circ\Do^{\chi,\eta}_{\Delta t/2}$, in which the exact penalisation
semigroup $\Do^{\chi,\eta}$ sandwiches a learnable multi-symplectic
evolution operator $\So_\theta$.  Such operators satisfy the conformal
symplectic conservation law
$\omega^{n+1}(\bx)=e^{-\alpha(\bx)\Delta t}\,\omega^{n}(\bx)$ with
$\alpha=\chi/\eta$, up to the exact discrete flux by
construction, for every admissible geometry $\chi$ and for every PDE
in the penalised class; the indicator function is an input, so a
single trained model covers families of domains.
\end{enumerate}
Numerical experiments in \cref{sec:experiments} illustrate each
component for both worked examples.  For the wave equation we verify
the convergence of the penalised model to the Dirichlet problem as
$\eta\to0$ and the second-order accuracy of the proposed integrator,
demonstrate its exact energy budget on two-dimensional scattering by
obstacles (where a non-geometric explicit baseline exhibits
catastrophic energy growth), and present an operator-learning
comparison on wave propagation past airfoil-shaped scatterers in which
the proposed conformal symplectic neural operator stays close to the discrete
energy of the ground truth where a standard FNO exhibits significant energy growth. For Maxwell's equations we verify the
expulsion of the electric field from penalised conductors as
$\eta\to0$, cross-validate the transverse magnetic mode against the
penalised wave equation (the two structure-preserving splittings are
observed to converge to each other at first order in $\Delta t$, the
formal order of their splitting-arrangement difference), confirm the
exact Poynting budget of the conformal split-Yee scheme on fully
three-dimensional scattering by a conducting sphere, and demonstrate
the electromagnetic CSNO on operator learning for three-dimensional
scattering by randomly sampled conducting obstacles, where it reduces
the rollout error of an FNO baseline by an order of
magnitude in comparison to the baseline which exhibits sharper energy growth.

\paragraph{Related work.}
Multi-symplectic structures of Hamiltonian PDEs and their
discretisations originate in
\cite{bridges1997multi,bridges2001multi,marsden1998multisymplectic};
box-type multi-symplectic schemes are analyzed in
\cite{bridges2001multi,moore2003backward}, and multi-symplectic
methods for Maxwell's equations in particular are studied in
\cite{sun2019maxwell}.  Conformal symplectic (linearly
damped Hamiltonian) systems were characterised by McLachlan and
Perlmutter \cite{mclachlan2001conformal}, and conformal conservation
laws with associated integrators for damped Hamiltonian PDEs with spatially constant damping 
were developed by Moore and collaborators
\cite{moore2013conformal,bhatt2016conformal}.  Our contribution develop on these works with in three ways. Firstly, we identify a general algebraic criterion under which partial damping (damping of a subset of components) is conformal after an explicit
modification of the Hamiltonian density. Second is the observation that the Brinkman penalisation method (which was originally conceived purely for computational convenience) is a member of this class with a spatially inhomogeneous damping coefficient $\alpha(\bx)=\chixx/\eta$. Finally, we exploit these properties for both numerical integration and operator learning.  On the learning side, structure-preserving architectures for finite-dimensional Hamiltonian systems include
Hamiltonian neural networks \cite{greydanus2019hamiltonian} and
SympNets \cite{jin2020sympnets}; neural operators for PDEs include
\cite{lu2021learning,li2021fourier,kovachki2023neural}.  The
architecture proposed here can be viewed as a conformal, infinite
dimensional analogue of SympNets adapted to penalised complex-domain
problems to allow flexible geometry adaptation. 

\section{Preliminaries}
\label{sec:preliminaries}

\subsection{Brinkman penalisation}
\label{subsec:brinkman}

Let $\Omega=\Omega_f\cup\Omega_s\subset\R^d$ be a simple computational
box decomposed into a fluid region $\Omega_f$ and a solid region
$\Omega_s$ with indicator function \eqref{eq:indicator}, and consider
an evolutionary PDE posed on $\Omega_f$ with homogeneous boundary
conditions on $\partial\Omega_f$.  Boundary-fitted discretisations
require mesh generation adapted to $\partial\Omega_f$, which is
difficult and costly for complex shapes and must be repeated after
every shape change. The same issue reappears in some scientific machine learning applications as the need to
retrain models for each geometry.  The Brinkman penalisation method
\cite{angot1999penalisation} instead solves an equation on the whole
box $\Omega$ in which the solid is modeled as a strongly dissipative
medium. A friction term of size $1/\eta$, $0<\eta\ll1$, supported on
$\Omega_s$ drives the penalised fields to zero there, and the boundary
condition on $\partial\Omega_f$ is recovered in the limit
$\eta\to0$. Rigorous convergence estimates for penalisations of this
type can be found in \cite{angot1999penalisation,kolomenskiy2009fourier}.
The penalised equation can be discretised on a typical Cartesian grid, and the domain geometry enters only through $\chi$, which may serve directly as an
input channel of a neural operator.  Two prototypical examples used throughout the
paper are the wave equation with Brinkman friction,
\begin{equation}
\label{eq:wave-brinkman}
u_{tt} \;=\; c^2\Delta u \;-\; \frac{\chixx}{\eta}\,u_t ,
\end{equation}
in which the friction damps the velocity $u_t$, and Maxwell's
equations penalised by an artificial Ohmic conductivity,
\begin{equation}
\label{eq:maxwell-brinkman}
\bE_t \;=\; \curl \bH \;-\; \frac{\chixx}{\eta}\,\bE,
\qquad
\bH_t \;=\; -\,\curl \bE ,
\end{equation}
in nondimensional units $\varepsilon=\mu=1$, in which the solid is
modeled as a conductor obeying Ohm's law $\bm{J}=\sigma\bE$ with
conductivity $\sigma=\chixx/\eta$. As $\eta\to0$ the electric field is
expelled from $\Omega_s$ and the perfect-conductor boundary condition
$\bm{n}\times\bE=0$ on $\partial\Omega_f$ is recovered.  In both
cases the damping acts on a subset of the dynamical fields
($u_t$ but not $u$; $\bE$ but not $\bH$), which is an important feature that enables the analysis below.

\subsection{Multi-symplectic Hamiltonian PDEs}
\label{subsec:ms}

A \emph{multi-symplectic Hamiltonian PDE}
\cite{bridges1997multi,bridges2001multi} in $d-$dimensional space is a
first-order system
\begin{equation}
\label{eq:ms}
\bK \bz_t + \sum_{k=1}^{d}\bL_k \bz_{x_k} = \nabla_{\bz} S(\bz),
\qquad \bz(t,\bx)\in\R^n,
\end{equation}
where $\bK,\bL_1,\dots,\bL_d\in\R^{n\times n}$ are constant
skew-symmetric matrices and $S:\R^n\to\R$ is a smooth scalar function
(possibly depending on $\bx$ through prescribed coefficients).
Solutions of \eqref{eq:ms} satisfy the multi-symplectic
conservation law
\begin{equation}
\label{eq:ms-law}
\partial_t\omega + \sum_{k=1}^{d}\partial_{x_k}\kappa_k = 0,
\qquad
\omega = \tfrac12\,\dd\bz\wedge\bK\,\dd\bz,
\qquad
\kappa_k = \tfrac12\,\dd\bz\wedge\bL_k\,\dd\bz,
\end{equation}
where $\dd\bz$ denotes the solution of the variational equation
associated with \eqref{eq:ms}.  All Hamiltonian PDEs, such as the wave equation, nonlinear Schr\"odinger, Korteweg--de Vries, and Maxwell equation all can be written in the form \eqref{eq:ms}. A numerical scheme is referred to as multi-symplectic if a discrete analogue of \eqref{eq:ms-law} holds exactly for the discrete variational equation. For example the Preissman box scheme and the Euler box scheme
\cite{bridges2001multi,moore2003backward} are multi-symplectic numerical schemes.

\begin{example}[Wave equation]
\label{ex:wave-ms}
With $v=u_t$ and $w=c\,u_x$, the one-dimensional wave equation
$u_{tt}=c^2u_{xx}$ takes the form \eqref{eq:ms} with
\begin{equation}
\label{eq:wave-ms-matrices}
\bz=\begin{pmatrix}u\\ v\\ w\end{pmatrix},
\quad
\bK=\begin{pmatrix}0&-1&0\\ 1&0&0\\ 0&0&0\end{pmatrix},
\quad
\bL=\begin{pmatrix}0&0&c\\ 0&0&0\\ -c&0&0\end{pmatrix},
\quad
S(\bz)=\tfrac12\bigl(v^2-w^2\bigr).
\end{equation}
Indeed, the three rows of \eqref{eq:ms} read
$-v_t+c\,w_x = 0$, $u_t = v$, and $-c\,u_x=-w$.  The associated
densities are $\omega=\dd u\wedge\dd v$ and
$\kappa=c\,\dd u\wedge\dd w$.  The multidimensional wave equation is
analogous, with $w_k=c\,u_{x_k}$ and one matrix $\bL_k$ per spatial
direction.
\end{example}

\begin{example}[Maxwell's equations]
\label{ex:maxwell-ms}
Maxwell's equations in vacuum, $\bE_t=\curl\bH$, $\bH_t=-\curl\bE$,
take the form \eqref{eq:ms} with $n=6$,
\begin{equation}
\label{eq:maxwell-ms-matrices}
\bz=\begin{pmatrix}\bE\\ \bH\end{pmatrix},
\qquad
\bK=\begin{pmatrix}0&\bI_3\\ -\bI_3&0\end{pmatrix},
\qquad
\bL_k=\begin{pmatrix}\bm{C}_k&0\\ 0&\bm{C}_k\end{pmatrix},
\qquad
S(\bz)\equiv 0 ,
\end{equation}
where $(\bm{C}_k)_{ij}=\epsilon_{ikj}$ are the (skew-symmetric)
generators of the curl, $\sum_k \bm{C}_k\partial_{x_k}\bm{F}
=\curl\bm{F}$; see also
\cite{sun2019maxwell}.  Indeed,
$\bK\bz_t=(\bH_t,-\bE_t)^{\mathsf T}$ and
$\sum_k\bL_k\bz_{x_k}=(\curl\bE,\curl\bH)^{\mathsf T}$, so
\eqref{eq:ms} with $S\equiv0$ reproduces the two curl equations.  The
symplectic density is $\omega=\sum_{i=1}^3\dd E_i\wedge\dd H_i$.
\end{example}

\subsection{Conformal symplectic systems and their PDE analogue}
\label{subsec:conformal}

Structure preservation for dissipative systems is captured by
the notion of conformal symplecticity.  Following McLachlan and
Perlmutter \cite{mclachlan2001conformal}, a \emph{conformal
symplectic} (linearly damped Hamiltonian) system is
\begin{equation}
\label{eq:conformal-ode}
\dot{\bz} = \bJ\nabla H(\bz) - \alpha\,\bz ,
\qquad
\bJ=\begin{pmatrix}0&\bI\\ -\bI&0\end{pmatrix},
\quad \bz=(q,p)^{\mathsf T}\in\R^{2n},
\end{equation}
whose flow contracts the symplectic form at a uniform exponential
rate,
\begin{equation}
\label{eq:conformal-law-ode}
\Lie_X\,\omega = -2\alpha\,\omega
\quad\Longrightarrow\quad
\omega(t) = e^{-2\alpha t}\,\omega(0),
\qquad \omega = \dd q\wedge \dd p ,
\end{equation}
where $X$ is the vector field of \eqref{eq:conformal-ode}.  Conformal
symplectic methods are one-step maps that reproduce the exact
contraction \eqref{eq:conformal-law-ode} discretely
\cite{mclachlan2001conformal,bhatt2016conformal}.  The PDE analogue,
introduced for damped Hamiltonian PDEs in
\cite{moore2013conformal,bhatt2016conformal}, is given in the following definition.

\begin{definition}[Multi-conformal symplectic PDE]
\label{def:mcs}
Let $a=a(\bx)\ge0$ be a prescribed damping coefficient.  A PDE of the
form
\begin{equation}
\label{eq:mcs}
\bK\bz_t + \sum_{k=1}^{d}\bL_k\bz_{x_k}
= \nabla_{\bz} \widetilde S(\bz) - \frac{a(\bx)}{2}\,\bK\bz ,
\end{equation}
with skew-symmetric $\bK$, $\bL_k$ and smooth $\widetilde S$ (possibly
$\bx$-dependent through prescribed coefficients), is called
\emph{multi-conformal symplectic}.
\end{definition}

\begin{proposition}[Multi-conformal symplectic conservation law]
\label{prop:mcs-law}
Solutions of \eqref{eq:mcs} satisfy
\begin{equation}
\label{eq:mcs-law}
\partial_t \omega + \sum_{k=1}^{d}\partial_{x_k} \kappa_k
= -\,a(\bx)\,\omega ,
\end{equation}
with $\omega,\kappa_k$ as in \eqref{eq:ms-law}.
\end{proposition}

\begin{proof}
The variational equation of \eqref{eq:mcs} reads
$\bK\,\dd\bz_t + \sum_k\bL_k\,\dd\bz_{x_k}
= D^2\widetilde S(\bz)\,\dd\bz - \tfrac{a}{2}\bK\,\dd\bz$.  Taking the
wedge product with $\dd\bz$ on the left and using skew-symmetry of
$\bK$ and $\bL_k$,
\[
\dd\bz\wedge\bK\,\dd\bz_t = \partial_t\bigl(\tfrac12\,
\dd\bz\wedge\bK\,\dd\bz\bigr)=\partial_t\omega,
\qquad
\dd\bz\wedge\bL_k\,\dd\bz_{x_k} = \partial_{x_k}\kappa_k .
\]
Since $D^2\widetilde S$ is symmetric,
$\dd\bz\wedge D^2\widetilde S\,\dd\bz=0$, while
$\dd\bz\wedge\bigl(-\tfrac{a}{2}\bK\,\dd\bz\bigr) = -a\,\omega$.
Combining the three identities yields \eqref{eq:mcs-law}.
\end{proof}

The conservation law \eqref{eq:mcs-law} is local, meaning the
symplectic density $\omega$ is transported conservatively where $a=0$
and decays exponentially at rate $a(\bx)$ where $a>0$.  This locality
is precisely what is needed for the Brinkman method, whose dissipation
is confined to the solid region. In fact, global notions, such as symplecticity of the time-$t$ flow are too coarse to capture the localised nature of the penalisation.

\section{Multi-conformal symplectic structure of Brinkman-penalised
Hamiltonian PDEs}
\label{sec:structure}
Brinkman penalisation was initially introduced as a convenient computational trick to allow the solver avoid mesh generation. To date, existing literature has not yet investigated whether the penalised equation retains any of the geometric structure of the Hamiltonian PDE it approximates. This section shows that we can in fact say more about the retained structure of Brinkman-penalised methods, under one simple algebraic condition. This interpretation allows us to construct appropriate numerical integrators for this structure in Section 4, and the structure preserving neural operator architectures of Section 5. We first show why the naive way of trying to establish this structure fails, and further go on to derive a corrected formulation of Brinkman penalisation as multi-conformal symplectic structure.
\subsection{The penalised class and a naive attempt}
\label{subsec:class}

Both examples of \cref{subsec:brinkman} penalise a subset of
the evolution equations of a multi-symplectic system. In the wave
equation the damping is in the momentum equation, and Maxwell's equations damp the Amp\`ere equation.  We formalise this as follows.  Let $\bP=\bP^{\mathsf
T}=\bP^2\in\R^{n\times n}$ be an orthogonal projection selecting the
penalised equations, and consider the Brinkman-penalised
multi-symplectic PDE
\begin{equation}
\label{eq:penalised}
\bK\bz_t + \sum_{k=1}^{d}\bL_k\bz_{x_k}
= \nabla_{\bz}S(\bz) \;-\; \frac{\chixx}{\eta}\,\bP\bK\bz .
\end{equation}
The friction term subtracts $(\chi/\eta)$ times the corresponding
components of $\bK\bz$ from the selected rows as verified in
\cref{subsec:wave,subsec:maxwell} below. This reproduces exactly the
physical penalisations \eqref{eq:wave-brinkman} and
\eqref{eq:maxwell-brinkman}.

Comparing \eqref{eq:penalised} with \cref{def:mcs}, one is tempted to
match the two damping terms by keeping $S$ untouched and setting
$a=\chi/\eta$ in \eqref{eq:mcs}, i.e., replacing
$\bP\bK\bz$ by $\tfrac12\bK\bz$.  This naive substitution changes the
dynamics. For example after eliminating $v$ and $w$ in the wave equation with the data \eqref{eq:wave-ms-matrices}, the naive system
$\bK\bz_t+\bL\bz_x=\nabla S(\bz)-\frac{\chi}{2\eta}\bK\bz$ reduces to
\begin{equation}
\label{eq:naive-reduced}
u_{tt} - c^2 u_{xx}
= -\,\frac{\chixx}{\eta}\,u_t
\;-\;\frac{\chixx^{2}}{4\eta^{2}}\,u ,
\end{equation}
which is different to the penalised wave equation
\eqref{eq:wave-brinkman}. We see that the correct Brinkman friction appears, but it is accompanied by an additional stiff restoring force
$-\chi^2 u/(4\eta^2)$ supported on the solid. 
That is, we find that the
canonical conformal damping $-\frac{a}{2}\bK\bz$ acts on all
multi-symplectic variables, whereas Brinkman-type friction acts on a
subset. This leads to a discrepancy of a linear term which, if it is a gradient, can be absorbed into the Hamiltonian density.

\subsection{Main structural theorem}
\label{subsec:main}

\begin{theorem}[Multi-conformal symplectic structure of
Brinkman-penalised multi-symplectic PDEs]
\label{thm:general}
Let $\bK,\bL_1,\dots,\bL_d$ be skew-symmetric, let $S$ be smooth, let
$\chi$ be time-independent, and let $\bP=\bP^{\mathsf T}=\bP^2$ be an
orthogonal projection satisfying the \emph{compatibility condition}
\begin{equation}
\label{eq:compatibility}
\bK \;=\; \bP\bK + \bK\bP .
\end{equation}
Then the matrix $\bM:=(\bI-2\bP)\bK$ is symmetric, and the
Brinkman-penalised system \eqref{eq:penalised} is identical to the
multi-conformal symplectic system
\begin{equation}
\label{eq:penalised-canonical}
\bK\bz_t + \sum_{k=1}^{d}\bL_k\bz_{x_k}
= \nabla_{\bz}S_\chi(\bz) \;-\; \frac{\chixx}{2\eta}\,\bK\bz ,
\qquad
S_\chi(\bz)
\;=\;
S(\bz) \;+\; \frac{\chixx}{4\eta}\,\bz^{\mathsf T}\bM\,\bz .
\end{equation}
Consequently, by \cref{prop:mcs-law} with $a=\chi/\eta$, solutions of
\eqref{eq:penalised} satisfy the locally conformal multi-symplectic
conservation law
\begin{equation}
\label{eq:brinkman-law}
\partial_t\omega + \sum_{k=1}^{d}\partial_{x_k}\kappa_k
\;=\;
-\,\frac{\chixx}{\eta}\,\omega .
\end{equation}
\end{theorem}

\begin{proof}
To see that symmetry of $\bM$ is equivalent to \eqref{eq:compatibility} we use $\bK^{\mathsf T}=-\bK$ and $\bP^{\mathsf T}=\bP$ and compute,
\[
\bM^{\mathsf T} = \bK^{\mathsf T}(\bI-2\bP) = -\bK+2\bK\bP,
\qquad
\bM = \bK - 2\bP\bK ,
\]
so $\bM=\bM^{\mathsf T}$ if and only if $2\bK = 2\bP\bK+2\bK\bP$.
Given \eqref{eq:compatibility}, decompose the friction term as
\[
-\frac{\chi}{\eta}\,\bP\bK\bz
= \frac{\chi}{2\eta}\,(\bK-2\bP\bK)\bz
  \;-\;\frac{\chi}{2\eta}\,\bK\bz
= \frac{\chi}{2\eta}\,\bM\bz \;-\;\frac{\chi}{2\eta}\,\bK\bz .
\]
Since $\bM$ is symmetric,
$\frac{\chi}{2\eta}\bM\bz
= \nabla_{\bz}\bigl(\frac{\chi}{4\eta}\bz^{\mathsf T}\bM\bz\bigr)$,
and absorbing this gradient into the density yields
\eqref{eq:penalised-canonical}.  The conservation law
\eqref{eq:brinkman-law} then follows from \cref{prop:mcs-law}. Note
that the $\bx$-dependence of $S_\chi$ through $\chi$ is inconsequential
because only the $\bz$-Hessian of the density enters the proof.
\end{proof}

In addition to the above Theorem, we now show that the compatibility condition admits an intuitive structural interpretation.

\begin{lemma}
\label{lem:interpretation}
For an orthogonal projection $\bP$ and skew-symmetric $\bK$, condition
\eqref{eq:compatibility} holds if and only if
\begin{equation}
\label{eq:block-offdiag}
\bP\bK\bP = 0
\qquad\text{and}\qquad
(\bI-\bP)\,\bK\,(\bI-\bP) = 0 ,
\end{equation}
that is, if and only if $\bK$ is off-diagonal with respect to the
splitting $\R^n=\ran\bP\oplus\ker\bP$.  In this case
$\bP\bK=\bK(\bI-\bP)$ and $\bK\bP=(\bI-\bP)\bK$.
\end{lemma}

\begin{proof}
If \eqref{eq:compatibility} holds, multiplying by $\bP$ on both sides
gives $\bP\bK\bP = \bP(\bP\bK+\bK\bP)\bP = 2\bP\bK\bP$, hence
$\bP\bK\bP=0$; similarly $(\bI-\bP)\bK(\bI-\bP)=0$ follows by
multiplying with $\bI-\bP$.  Conversely, expanding
$\bK=(\bP+(\bI-\bP))\bK(\bP+(\bI-\bP))$ and deleting the two vanishing
blocks yields $\bK=\bP\bK(\bI-\bP)+(\bI-\bP)\bK\bP=\bP\bK+\bK\bP$
(using \eqref{eq:block-offdiag} again).  The final identities follow
from $\bP\bK=\bP\bK(\bI-\bP)=\bK(\bI-\bP)$ under
\eqref{eq:block-offdiag}.
\end{proof}

This says the penalised equations must be coupled by the
symplectic structure only to unpenalised ones, or in other words each damped variable is $\bK$-conjugate to an undamped variable.  Damping momenta while
leaving positions free, or damping electric fields while leaving
magnetic fields free, is of exactly this type, but damping a canonically
conjugate pair jointly is not (see e.g. Remark \ref{rem:nls}).  For the canonical
structure $\bK=\bJ$ with $\bz=(q,p)$ and $\bP$ the projection onto the
momentum equations, \eqref{eq:compatibility} is verified immediately.

\begin{remark}
\label{rem:nls}
If the penalisation damps complete $\bK$-conjugate pairs (as is the case for the linearly damped nonlinear Schr\"odinger equation, where the friction $-\gamma\psi$ acts on both the real and imaginary parts of $\psi$) then the friction term is proportional to $\bK\bz$ itself and
the system is already in the canonical form \eqref{eq:mcs} with
unmodified $S$. This class is studied in
\cite{moore2013conformal}.  \Cref{thm:general} addresses the
complementary case of partial damping, which is the typical
situation for Brinkman-type penalisations, and shows that the price of
partial damping is exactly the explicit quadratic correction
$\frac{\chi}{4\eta}\bz^{\mathsf T}\bM\bz$ of the Hamiltonian density,
supported on the solid region.
\end{remark}

\subsection{Example I: the wave equation with Brinkman friction}
\label{subsec:wave}

Take the multi-symplectic data \eqref{eq:wave-ms-matrices} of
Example \ref{ex:wave-ms} and let $\bP=\diag(1,0,0)$ select the first row,
i.e., the momentum equation $-v_t+c\,w_x=\dots$, which is where the
Brinkman friction acts.  Then
\[
\bP\bK+\bK\bP
= \begin{pmatrix}0&-1&0\\0&0&0\\0&0&0\end{pmatrix}
+ \begin{pmatrix}0&0&0\\1&0&0\\0&0&0\end{pmatrix}
= \bK ,
\qquad
\bM=(\bI-2\bP)\bK
= \begin{pmatrix}0&1&0\\1&0&0\\0&0&0\end{pmatrix},
\]
so \cref{thm:general} applies with
$\tfrac14\bz^{\mathsf T}\bM\bz=\tfrac12\,uv$.

\begin{corollary}[Wave equation]
\label{cor:wave}
With $v=u_t$, $w=c\,u_x$, the Brinkman-penalised wave equation
\eqref{eq:wave-brinkman} (in one space dimension; the multidimensional
case is identical) is the multi-conformal symplectic system
\eqref{eq:penalised-canonical} with
\begin{equation}
\label{eq:wave-Schi}
S_\chi(\bz)
= \tfrac12\bigl(v^2-w^2\bigr)
+ \frac{\chixx}{2\eta}\,u\,v ,
\end{equation}
and its solutions satisfy
$\partial_t\omega+\partial_x\kappa=-\frac{\chi}{\eta}\omega$ with
$\omega=\dd u\wedge\dd v$ and $\kappa=c\,\dd u\wedge\dd w$.
\end{corollary}

\begin{proof}
Only the identification of \eqref{eq:penalised} with
\eqref{eq:wave-brinkman} remains to be checked.  Since
$\bP\bK\bz=(-v,0,0)^{\mathsf T}$, the rows of \eqref{eq:penalised}
read $-v_t+c\,w_x=(\chi/\eta)v$, $u_t=v$, $w=c\,u_x$; eliminating $v$
and $w$ gives $u_{tt}=c^2u_{xx}-(\chi/\eta)u_t$.
\end{proof}

\begin{remark}
\label{rem:shift}
The quadratic correction in \eqref{eq:wave-Schi} is a solid-supported
coupling between position and velocity.  It can be traded for a
solid-supported potential by the linear change of variables
$\tilde v = v + \frac{\chi}{2\eta}u$ (this is the PDE analogue of the
classical transformation that converts a damped oscillator into
conformal Hamiltonian form). The shear
$T=\bigl(\begin{smallmatrix}1&0&0\\ \chi/2\eta&1&0\\
0&0&1\end{smallmatrix}\bigr)$ satisfies $T^{\mathsf T}\bK T=\bK$ and
$T^{\mathsf T}\bL T=\bL$, i.e., it preserves the multi-symplectic
structure, and transforms \eqref{eq:wave-Schi} into
\begin{equation}
\label{eq:wave-Schi-shift}
\widetilde S_\chi(\tilde\bz)
= \tfrac12\bigl(\tilde v^2 - w^2\bigr)
- \frac{\chixx^{2}}{8\eta^{2}}\,u^{2},
\qquad
\tilde v = u_t + \frac{\chixx}{2\eta}\,u .
\end{equation}
The correction term may then be interpreted as the potential energy
stored by the penalisation medium.  Both interpretations carry the same
conformal law \eqref{eq:brinkman-law}. We work with
\eqref{eq:wave-Schi} because it requires no change of variables and is
the direct output of \cref{thm:general}.
\end{remark}

\subsection{Example II: Maxwell's equations with conductivity
penalisation}
\label{subsec:maxwell}

Take the multi-symplectic data \eqref{eq:maxwell-ms-matrices} of
Example \ref{ex:maxwell-ms} and let $\bP=\diag(0,\bI_3)$ select the second
block row, i.e., the Amp\`ere equations
$-\bE_t+\curl\bH=\dots$ where the conduction current acts.
Then
\[
\bP\bK+\bK\bP
= \begin{pmatrix}0&0\\-\bI_3&0\end{pmatrix}
+ \begin{pmatrix}0&\bI_3\\0&0\end{pmatrix}
= \bK ,
\qquad
\bM=(\bI-2\bP)\bK
= \begin{pmatrix}0&\bI_3\\ \bI_3&0\end{pmatrix},
\]
so \cref{thm:general} applies with
$\tfrac14\bz^{\mathsf T}\bM\bz=\tfrac12\,\bE\cdot\bH$.

\begin{corollary}[Maxwell's equations]
\label{cor:maxwell}
The conductivity-penalised Maxwell system
\eqref{eq:maxwell-brinkman} is the multi-conformal symplectic system
\eqref{eq:penalised-canonical} with
\begin{equation}
\label{eq:maxwell-Schi}
S_\chi(\bz) \;=\; \frac{\chixx}{2\eta}\,\bE\cdot\bH ,
\end{equation}
and its solutions satisfy
\begin{equation}
\label{eq:maxwell-law}
\partial_t\omega+\sum_{k=1}^{3}\partial_{x_k}\kappa_k
= -\frac{\chixx}{\eta}\,\omega,
\qquad
\omega=\sum_{i=1}^{3}\dd E_i\wedge\dd H_i,
\quad
\kappa_k=\tfrac12\,\dd\bz\wedge\bL_k\,\dd\bz .
\end{equation}
\end{corollary}

\begin{proof}
Since $\bP\bK\bz=(0,-\bE)^{\mathsf T}$, the block rows of
\eqref{eq:penalised} read $\bH_t+\curl\bE=0$ and
$-\bE_t+\curl\bH=(\chi/\eta)\bE$, i.e., exactly
\eqref{eq:maxwell-brinkman}. Alternatively one can verify directly that
$\nabla_{\bz}S_\chi=(\frac{\chi}{2\eta}\bH,\frac{\chi}{2\eta}\bE)$ and
$-\frac{\chi}{2\eta}\bK\bz=(-\frac{\chi}{2\eta}\bH,
\frac{\chi}{2\eta}\bE)$ sum to $(0,(\chi/\eta)\bE)$.
\end{proof}

It is noteworthy that for Maxwell's equations no change of
variables is required at all. The entire modification is the
solid-supported bilinear density $\frac{\chi}{2\eta}\bE\cdot\bH$.  The
two examples demonstrate that the
general correction $\frac{\chi}{4\eta}\bz^{\mathsf T}\bM\bz$ may
manifest as a position--velocity coupling (wave), as a field--field
coupling (Maxwell), or, after a structure-preserving shear, as a pure
potential (\cref{rem:shift}).

\begin{remark}
\label{rem:tm}
In two dimensions the transverse magnetic (TM) mode of
\eqref{eq:maxwell-brinkman}, with fields $(E_z,H_x,H_y)$ independent
of $z$, reduces to the two-dimensional penalised wave equation
\eqref{eq:wave-brinkman} for $u=E_z$.  The wave-equation experiments
of \cref{sec:experiments} therefore double as experiments on
two-dimensional electromagnetic scattering by conducting obstacles.
\end{remark}

\subsection{Physical meaning of the conservation law}
\label{subsec:physical}

The identity \eqref{eq:brinkman-law} decomposes the dynamics of the
symplectic density according to the geometry:
\begin{itemize}[leftmargin=2em]
\item \textbf{Fluid domain} ($\chi=0$):
$\partial_t\omega+\sum_k\partial_{x_k}\kappa_k=0$; the system is
conservative and waves propagate without loss, much like an
unpenalised Hamiltonian PDE.
\item \textbf{Solid domain} ($\chi=1$): the symplectic density decays
as $\omega\sim e^{-t/\eta}$ for some
fast time scale $\eta$. This is the geometric mechanism by which the
penalisation expels the fields from $\Omega_s$.
\end{itemize}
The same dichotomy holds for the energy.  For the wave equation,
multiplying \eqref{eq:wave-brinkman} by $u_t$ yields the local energy
identity
\begin{equation}
\label{eq:energy-identity}
\partial_t E + \nabla\!\cdot\!F = -\,\frac{\chixx}{\eta}\,u_t^2,
\qquad
E=\tfrac12\bigl(u_t^2 + c^2|\nabla u|^2\bigr),
\quad
F = -\,c^2\,u_t\,\nabla u ,
\end{equation}
and for Maxwell's equations Poynting's theorem gives
\begin{equation}
\label{eq:poynting}
\partial_t\Bigl(\tfrac12|\bE|^2+\tfrac12|\bH|^2\Bigr)
+ \nabla\!\cdot\!\bigl(\bE\times\bH\bigr)
= -\,\frac{\chixx}{\eta}\,|\bE|^2 ,
\end{equation}
where the right-hand side is the Ohmic dissipation
$-\bm{J}\cdot\bE$ of the artificial conductor.  In both cases the
energy is conserved locally in $\Omega_f$ and dissipated only in
$\Omega_s$.  In the following \cref{sec:integrators,sec:operators} we propose numerical methods and learned surrogates that reproduce
this behaviour by construction. 
\section{Structure-preserving integrators for the penalised class}
\label{sec:integrators}

\subsection{Strang splitting into exact dissipation and
multi-symplectic transport}
\label{subsec:splitting}

Conformal symplectic integrators for linearly damped systems are
classically obtained by splitting the vector field into its
Hamiltonian and dissipative parts and composing the exact dissipative
flow with a symplectic integrator
\cite{mclachlan2001conformal,mclachlan2002splitting,bhatt2016conformal}.
The structure identified in \cref{thm:general} makes the same strategy
available for the entire Brinkman-penalised class \eqref{eq:penalised}.  First we split
\eqref{eq:penalised} as
\begin{equation}
\label{eq:splitting}
\underbrace{\;\bK\bz_t + \sum_{k}\bL_k\bz_{x_k} = \nabla S(\bz)\;}
_{\text{(C): conservative multi-symplectic PDE}}
\qquad\qquad
\underbrace{\;\bK\bz_t = -\,\frac{\chixx}{\eta}\,\bP\bK\bz\;}
_{\text{(D): dissipative part}} .
\end{equation}
By \cref{lem:interpretation}, $\bP\bK=\bK(\bI-\bP)$, so subsystem (D)
is solved by damping the unprojected components using exact damping of the form
\begin{equation}
\label{eq:dissipation-semigroup}
\Do^{\chi,\eta}_{\tau}
\;=\;
\bP' \;+\; e^{-\chixx\,\tau/\eta}\,\bigl(\bI-\bP'\bigr),
\qquad
\bP' := \bP \ \text{on } \ran\bK,
\quad \bI \ \text{on } \ker\bK.
\end{equation}
Here the fields $(\bI-\bP)\bz$ lying in $\ran\bK$ are multiplied by
by $e^{-\chi\tau/\eta}$ pointwise, and the remaining components are kept fixed.
For the wave equation \eqref{eq:wave-ms-matrices} this is the velocity
damping $v\mapsto e^{-\chi\tau/\eta}v$ with $u$ and $w$ frozen; for
Maxwell's equations \eqref{eq:maxwell-ms-matrices} it is the field
damping $\bE\mapsto e^{-\chi\tau/\eta}\bE$ with $\bH$ frozen.
Subsystem (C) is the unpenalised Hamiltonian PDE, for which any
multi-symplectic method may be used. In our experiments we choose the
leapfrog/Euler-box discretisation
\cite{bridges2001multi,moore2003backward}, whose one-step map we
denote $\Phi^{\mathrm{MS}}_{\Delta t}$.  The proposed scheme is the
Strang composition
\begin{equation}
\label{eq:scheme}
\Phi_{\Delta t}
\;=\;
\Do^{\chi,\eta}_{\Delta t/2}
\;\circ\;
\Phi^{\mathrm{MS}}_{\Delta t}
\;\circ\;
\Do^{\chi,\eta}_{\Delta t/2}.
\end{equation}
Being a Strang splitting of two flows, each realised exactly or by a
second-order symmetric method, $\Phi_{\Delta t}$ is second-order
accurate \cite{strang1968construction,mclachlan2002splitting}. Since
the stiff factor $\chi/\eta$ is confined to the exactly integrated
substep, the splitting suffers no stability restriction from the
penalty parameter, and the CFL condition of the conservative substep
is the only step-size constraint.

Under the compatibility condition, the dissipative substep carries
exactly the conformal contraction of the symplectic density,
while the conservative substep carries the flux. This is made precise in the following lemma.

\begin{lemma}
\label{lem:dissipative-contraction}
Assume \eqref{eq:compatibility} and let
$T_a=\bP+e^{-a}(\bI-\bP)$, $a\ge0$.  Then
\begin{equation}
\label{eq:TKT}
T_a^{\mathsf T}\,\bK\,T_a \;=\; e^{-a}\,\bK ,
\end{equation}
so the tangent map of $\Do^{\chi,\eta}_{\tau}$ scales the symplectic
density pointwise by the exact factor:
$\omega\mapsto e^{-\chixx\,\tau/\eta}\,\omega$.  (Components in
$\ker\bK$ do not contribute to $\omega$, so the modification of $T_a$
on $\ker\bK$ in \eqref{eq:dissipation-semigroup} is inconsequential.)
\end{lemma}

\begin{proof}
Expanding and using \cref{lem:interpretation},
\[
T_a^{\mathsf T}\bK T_a
= \bP\bK\bP + e^{-a}\bigl(\bP\bK(\bI-\bP)+(\bI-\bP)\bK\bP\bigr)
+ e^{-2a}(\bI-\bP)\bK(\bI-\bP)
= e^{-a}\,\bK ,
\]
since the first and last terms vanish by \eqref{eq:block-offdiag} and
the middle bracket equals $\bK$.  Then
$\tfrac12(T_a\dd\bz)\wedge\bK(T_a\dd\bz)
=\tfrac12\,\dd\bz\wedge T_a^{\mathsf T}\bK T_a\,\dd\bz
=e^{-a}\omega$.
\end{proof}

\subsection{Discrete conformal multi-symplectic conservation law}
\label{subsec:discrete-law}

Let $\omega^n_i$ denote the discrete symplectic density associated
with the discrete variational equation of the composed scheme at the
grid point $(t_n,\bx_i)$, and suppose that
$\Phi^{\mathrm{MS}}_{\Delta t}$ satisfies the discrete
multi-symplectic conservation law
\begin{equation}
\label{eq:discrete-ms}
\frac{\tilde\omega^{\,n+1}_i - \tilde\omega^{\,n}_i}{\Delta t}
+ \sum_{k=1}^{d}
\bigl(\delta_k\,\tilde\kappa^{\,n+1/2}_k\bigr)_i = 0
\end{equation}
for the intermediate variables $\tilde{\bz}$ to which it is applied.
Here $\delta_k$ denotes the scheme's spatial difference operator in
direction $k$ (this holds for the Preissman and Euler box schemes
\cite{bridges2001multi}).

\begin{proposition}
\label{prop:discrete-law}
Assume \eqref{eq:compatibility} and that $\chi$ is constant on the
stencil of grid point $i$ and write $\alpha_i=\chi(\bx_i)/\eta$.  Then
the scheme \eqref{eq:scheme} satisfies
\begin{equation}
\label{eq:discrete-conformal}
\frac{\omega^{\,n+1}_i - e^{-\alpha_i\Delta t}\,\omega^{\,n}_i}
     {\Delta t}
\;+\;
e^{-\alpha_i\Delta t/2}\,
\sum_{k=1}^{d}\bigl(\delta_k\,\tilde\kappa^{\,n+1/2}_k\bigr)_i
\;=\;0 ,
\end{equation}
where $\tilde\kappa_k$ are the discrete fluxes of the interior
multi-symplectic scheme evaluated at the intermediate stage.  This is
the exact discrete analogue of \eqref{eq:brinkman-law}. In the fluid
region ($\alpha_i=0$) the scheme is multi-symplectic, and in the
interior of the solid region the discrete symplectic density contracts
by the exact factor $e^{-\Delta t/\eta}$ per step, up to the discrete
flux.
\end{proposition}

\begin{proof}
By \cref{lem:dissipative-contraction}, the first half-step gives
$\tilde\omega^{\,n}_i = e^{-\alpha_i\Delta t/2}\,\omega^n_i$; the
interior step obeys \eqref{eq:discrete-ms}; the final half-step gives
$\omega^{n+1}_i=e^{-\alpha_i\Delta t/2}\,\tilde\omega^{\,n+1}_i$.
Substituting the first and third relations into
\eqref{eq:discrete-ms} and multiplying through by
$e^{-\alpha_i\Delta t/2}$ yields \eqref{eq:discrete-conformal}.
\end{proof}

\begin{corollary}
\label{cor:discrete-wave}
For the wave equation \eqref{eq:wave-ms-matrices} the flux density
$\kappa=c\,\dd u\wedge\dd w$ involves only fields fixed by
$\Do^{\chi,\eta}$, so
$\tilde\kappa^{\,n+1/2}=\kappa^{\,n+1/2}$ and
\eqref{eq:discrete-conformal} holds with the fluxes expressed directly
in the original variables.  For Maxwell's equations the fluxes
$\kappa_k$ mix damped and undamped fields and the intermediate-stage
formulation of \cref{prop:discrete-law} is the natural one.
\end{corollary}

\begin{remark}
Near $\partial\Omega_s$, where $\chi$ jumps across the stencil, the
law \eqref{eq:discrete-conformal} holds with $\alpha_i$ replaced by
the local grid values; the statement ``discretely preserved'' is thus
to be understood either as an exact finite-difference identity of the
above type or, for smoothed indicator functions, in the sense of
backward error analysis \cite{moore2003backward,hairer2006geometric}.
An analogous computation shows that the scheme reproduces the discrete
counterparts of the energy identities
\eqref{eq:energy-identity}--\eqref{eq:poynting}. No energy is created
or destroyed in the fluid region beyond the exact transport of the
multi-symplectic interior scheme, and dissipation occurs only through
the exact factors $e^{-\chi\Delta t/(2\eta)}$ acting on the damped
fields.
\end{remark}

\section{Conformal symplectic neural operators}
\label{sec:operators}

\subsection{From SympNets to structure-preserving operator learning}

SympNets \cite{jin2020sympnets} approximate symplectic maps on
$\R^{2n}$ by composing elementary symplectic layers, e.g., shear maps
generated by gradient potentials; every network in the class is
symplectic by construction, and the class is dense in the set of
symplectic maps.  Lifting this idea from finite-dimensional phase
spaces to function spaces yields \emph{symplectic neural operators} \cite{makara2026symplecticneuraloperatorslearning}, where the state vector $\bz\in\R^{2n}$ is replaced by a field
$\bz(\bx)$, symplectic maps are replaced by multi-symplectic evolution
operators, and the layers are chosen so that the multi-symplectic
conservation law $\partial_t\omega+\nabla\!\cdot\!\kappa=0$ is preserved exactly.
More concretely, if we write the fields in $\bK$-conjugate pairs
$(q,p)$ (i.e. $(u,v)$ for the wave equation, $(\bH,\bE)$ for Maxwell's
equations) a layer of the operator $\So_{\theta,\Delta t}$ alternates
kick and drift shears
\begin{equation}
\label{eq:sympnet-layers}
\begin{pmatrix} q \\ p\end{pmatrix}
\mapsto
\begin{pmatrix} q \\ p - \Delta t\,\nabla_q V_\theta[q]\end{pmatrix},
\qquad
\begin{pmatrix} q \\ p\end{pmatrix}
\mapsto
\begin{pmatrix} q + \Delta t\,\nabla_p T_\theta[p]\\
p\end{pmatrix},
\end{equation}
where $V_\theta$ and $T_\theta$ are learnable functionals (e.g.,
integral functionals whose densities are parameterised by neural
networks acting on the fields) and the gradients are
$L^2$ functional derivatives.  Each shear is the exact flow of a
Hamiltonian system and is therefore multi-symplectic and the composition
of such steps maintains this structure.  This mirrors the way symplectic integrators serve as building blocks for structure-preserving networks in finite dimensions.

\subsection{Architecture}

For the Brinkman-penalised class \eqref{eq:penalised} the appropriate
invariance is the conformal law \eqref{eq:brinkman-law}.  Guided by
the splitting \eqref{eq:scheme}, we define the \emph{conformal
symplectic neural operator} (CSNO)
\begin{equation}
\label{eq:csno}
\Go_\theta
\;=\;
\Do^{\chi,\eta}_{\Delta t/2}
\;\circ\;
\So_{\theta,\Delta t}
\;\circ\;
\Do^{\chi,\eta}_{\Delta t/2},
\end{equation}
where $\Do^{\chi,\eta}_{\tau}$ is the exact penalisation
semigroup \eqref{eq:dissipation-semigroup}; that is, a fixed, parameter-free multiplication operator depending on the geometry indicator function $\chi$ that damps the fields selected by the structure of the PDE ($u_t$ for the wave equation, $\bE$ for
Maxwell's equations). The $\So_{\theta,\Delta t}$ component is a learnable multi-symplectic evolution operator built from layers
\eqref{eq:sympnet-layers}.  The geometry can additionally be supplied
to $\So_\theta$ through conditioning of $V_\theta$, $T_\theta$ on
$\chi$, in which case multi-symplecticity must be maintained by
letting $\chi$ enter only the potentials, not the symplectic
structure.

\begin{proposition}[Exact conformal law of the CSNO]
\label{prop:csno-law}
Assume \eqref{eq:compatibility}.  For every parameter $\theta$ and
every admissible indicator function $\chi$, the operator
\eqref{eq:csno} satisfies pointwise
\begin{equation}
\label{eq:csno-law}
\omega^{\,n+1}(\bx) = e^{-\alpha(\bx)\Delta t}\,\omega^{\,n}(\bx)
\;-\;\Delta t\, e^{-\alpha(\bx)\Delta t/2}\,
\nabla\!\cdot\!\tilde\kappa^{\,n+1/2}(\bx),
\qquad
\alpha(\bx)=\frac{\chixx}{\eta},
\end{equation}
where $\tilde\kappa^{n+1/2}$ is the flux of the interior
multi-symplectic operator at the intermediate stage; in particular the
spatially integrated symplectic density obeys the exact contraction
dictated by the penalisation, independently of training.
\end{proposition}

The proof is identical to that of \cref{prop:discrete-law}, with
$\Phi^{\mathrm{MS}}_{\Delta t}$ replaced by $\So_{\theta,\Delta t}$
and \cref{lem:dissipative-contraction} applied verbatim, though there are a few points worth emphasizing.  First, the conformal law
\eqref{eq:csno-law} is a hard architectural constraint that cannot be violated by optimisation error, distribution shift at inference time, or any other adversarial strategy. Second, because the geometry enters $\Do^{\chi,\eta}$ exactly and $\chi$ is an input rather than a training-time constant, a single trained CSNO applies to arbitrary domain shapes within the resolution of the grid, addressing the complex-domain challenge of \cref{sec:intro} uniformly across the penalised class. Third, the model and data share their geometric invariants by construction, since the architecture composes with the integrator of \cref{sec:integrators} in the data pipeline. However, in the case of the CSNO model the conservative substeps are learnable from data.

\begin{remark}
Density results for SympNets \cite{jin2020sympnets} suggest that
compositions of layers \eqref{eq:sympnet-layers} are dense, in
appropriate topologies, in classes of multi-symplectic evolution
operators. We expect that this property generalises, however a rigorous approximation theory in the operator setting is beyond the scope of this paper and is the subject of ongoing work. Combined with the exactness of the dissipative factor, this property would yield universal approximation of the penalised solution operator within the conformal class.  
\end{remark}

\section{Numerical experiments}
\label{sec:experiments}

We demonstrate some numerical results using our proposed integrators and neural operators. All experiments below use Cartesian grids on rectangular boxes and the domain geometry enters exclusively through the indicator function $\chi$. The computational box carries homogeneous Dirichlet (perfect electrical conductor condition for Maxwell) boundary conditions unless otherwise specified. We treat the wave equation in
\Cref{subsec:exp-1d,subsec:exp-2d,subsec:exp-op}, while \cref{subsec:exp-maxwell} treats Maxwell's equations,
including the fully three-dimensional conductivity penalisation.  

\subsection{Validation of the penalisation: 1D reflection}
\label{subsec:exp-1d}

We first verify the qualitative behavior of the penalised model
\eqref{eq:wave-brinkman} in one dimension (with $c=1$, discretised over $2048$ grid points with a Strang scheme \eqref{eq:scheme} at half the CFL step).  A Gaussian pulse is
launched toward a solid object $\Omega_s=[2,4]$ inside the box
$\Omega=[-6,6]$ and the solution at $T=5.0$ is compared for $\eta\in\{0.1,\,0.02,\,0.005\}$ in \cref{fig:brinkman1d}.  As $\eta$ decreases, an increasing fraction of the incident wave is reflected by the solid domain, which is representative of the approximation of the rigid-wall boundary condition improving as penalisation gets stronger. The
maximal field penetration $\max_{\Omega_s}|u(\cdot,T)|$ decreases
monotonically from $9.2\times10^{-2}$ ($\eta=0.1$) through
$5.1\times10^{-2}$ ($\eta=0.02$) to $2.7\times10^{-2}$
($\eta=0.005$), consistent with the known $O(\eta^{1/2})$-type
convergence of Brinkman-penalised models
\cite{angot1999penalisation,kolomenskiy2009fourier}.  A
self-convergence study in the time step on the same grid
(with $\eta=0.02$) 
yields the error
ratios $2.02$, $2.07$, $2.32$ under successive halvings of $\Delta
t$, confirming the second-order accuracy of the splitting
\eqref{eq:scheme} claimed in \cref{subsec:splitting}.

\begin{figure}[t]
\centering
\includegraphics[width=0.72\linewidth]{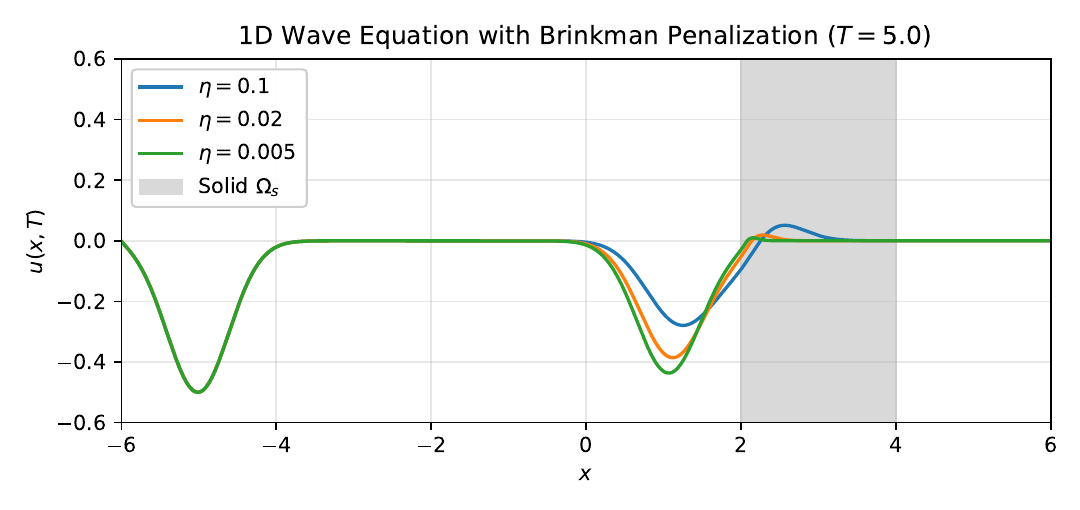}
\caption{One-dimensional wave equation with Brinkman penalisation at
$T=5.0$ for $\eta=0.1$, $0.02$, $0.005$ (solid region shaded).
Smaller $\eta$ yields stronger reflection; $\eta=0.005$ approximates a
rigid wall almost perfectly.}
\label{fig:brinkman1d}
\end{figure}

\subsection{Structure-preserving integration: 2D scattering by a
circular obstacle}
\label{subsec:exp-2d}

We next apply the splitting scheme \eqref{eq:scheme} to the
two-dimensional acoustic wave equation with a circular obstacle
($\eta=0.02$, $256^2$ grid on $[-4,4]^2$). 
In this experiment a Gaussian pulse is launched at $t=0$ which collides with the obstacle and produces a reflected and diffracted wave
patterns. We show snapshots at $t=0,\,1.5,\,3.0,\,4.0$ in
\cref{fig:wave2d}.  Waves penetrating the solid domain vanish
rapidly, on the time scale $\eta$, in accordance with the local
contraction \eqref{eq:brinkman-law}.

\Cref{fig:wave2d-energy} quantifies the energy behavior. 
Here we compare the discrete energy $H(t)$ of the splitting scheme with RK4 and explicit Euler applied to the full damped system on the same grid and time step, together with the
exact $H(0)-\int_0^t\!\int(\chi/\eta)\,u_t^2$, the discrete
counterpart of \eqref{eq:energy-identity}.  The split scheme tracks
the exact energy throughout. We validate this through the residual
$|H(T)+D(T)-H(0)|/H(0)=8.7\times10^{-3}$, where $D$ is the trapezoidal
time-quadrature of the instantaneous dissipation, which is at the level of the quadrature. We note that in this example RK4 attains a comparable
$4.1\times10^{-3}$ at this step size, though without the structural
guarantee of \cref{prop:discrete-law}, whereas explicit Euler
exhibits catastrophic energy growth (about one hundred orders of
magnitude by $T=6$).  Additionally, because the
stiff factor $\chi/\eta$ is confined to the exactly integrated
substep, the splitting scheme remains stable uniformly in $\eta$, while
any explicit one-step method applied to the full system is
constrained by $\Delta t/\eta$ lying in its stability region and
degrades as $\eta\to0$.

\begin{figure}[t]
\centering
\includegraphics[width=\linewidth]{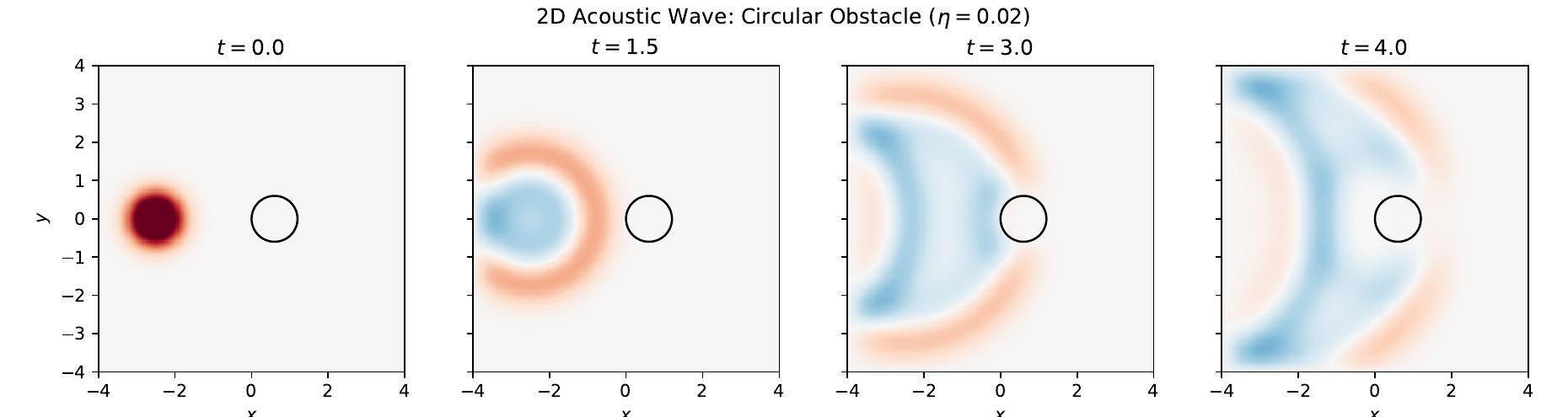}
\caption{2D acoustic wave scattering by a circular obstacle
($\eta=0.02$) computed with the conformal multi-symplectic split
scheme \eqref{eq:scheme}: snapshots at $t=0$, $1.5$, $3.0$, $4.0$.
The pulse is reflected by the obstacle while the field inside the
solid decays on the fast time scale $\eta$.}
\label{fig:wave2d}
\end{figure}

\begin{figure}[t]
\centering
\includegraphics[width=0.72\linewidth]{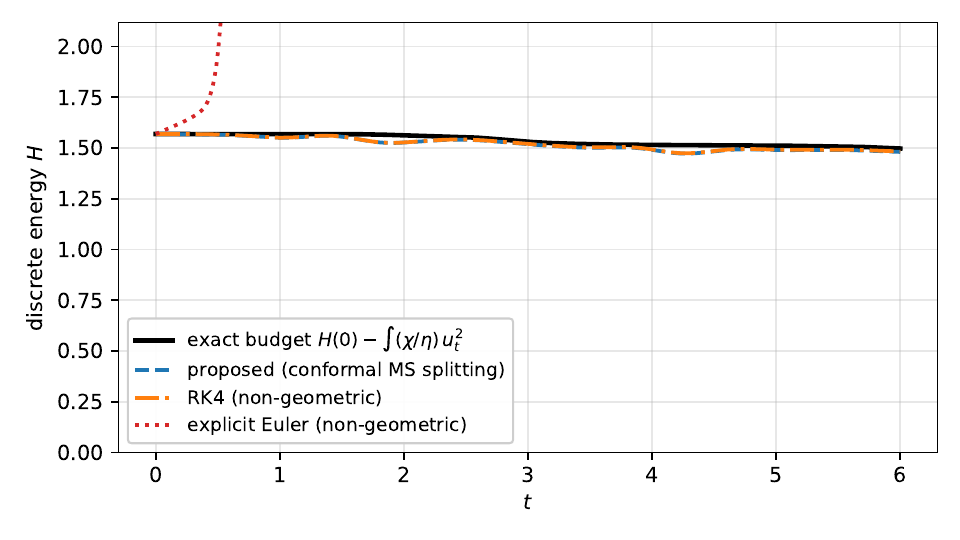}
\caption{Discrete energy $H(t)$ for the 2D scattering problem of
\cref{fig:wave2d}: exact budget
$H(0)-\int(\chi/\eta)u_t^2$ (black), the proposed conformal
multi-symplectic splitting, RK4, and explicit Euler on the same grid
and time step. The proposed scheme tracks the exact budget whereas explicit Euler exhibits catastrophic energy growth (leaving the frame).}
\label{fig:wave2d-energy}
\end{figure}

\subsection{Operator learning: wave propagation past an airfoil}
\label{subsec:exp-op}

Finally we evaluate the CSNO \eqref{eq:csno} on an operator-learning
task. Here we predict the evolution of the penalised wave field in the
presence of airfoil-shaped scatterers, with the geometry modelled
through $\chi$.  The dataset consists of trajectories past randomly
sampled symmetric (NACA) airfoils (these are randomly positioned with maximum thickness ratio in $[0.08,0.18]$, chord in $[1.2,2.2]$, angle of attack in $[-25^\circ,25^\circ]$) excited by random Gaussian
pulses. We integrate these with the structure-preserving scheme of
\cref{sec:integrators} so that the trajectories satisfy appropriate conformal and energy laws by construction, and use these as training data.  In the CSNO, the interior
operator $\So_\theta$ instantiates the layers
\eqref{eq:sympnet-layers} with convolutional potential densities
conditioned on $\chi$, preceded by the exact St\"ormer--Verlet step
of the known wave Hamiltonian as a fixed leading layer
and learnable correction layers initialized near the identity. The
dissipative factors are the exact multiplication operators of
\eqref{eq:csno}.  We compare against a Fourier neural operator
baseline \cite{li2021fourier} trained on the same one-step data with
the same inputs $(u,v,\chi)$ and the same loss. In both cases the solutions are rolled out over equal number of timesteps to produce the solution estimate.

\Cref{fig:operator} shows a representative rollout prediction on a
held-out geometry together with the evolution of the discrete energy
$H$.  Both models capture the gross features of the scattered field,
but their spectral and energy behavior diverge over time. The energy of the FNO prediction increases and it produces unnatural oscillations by time $t=4$. The CSNO on the other hand, maintains solution estimates and energy that closely tracks the ground truth over the entire prediction window. This is likely the operator-learning analogue of the numerical divergence observed with Euler style solvers of the previous subsection. That is, without an architectural guarantee, the learned operator introduces spurious growth in the fluid region, while the conformal symplectic architecture controls energy in the fluid region while confining dissipation to the solid region at the rate prescribed by the penalisation.

\begin{figure}[t]
\centering
\includegraphics[width=0.9\linewidth]{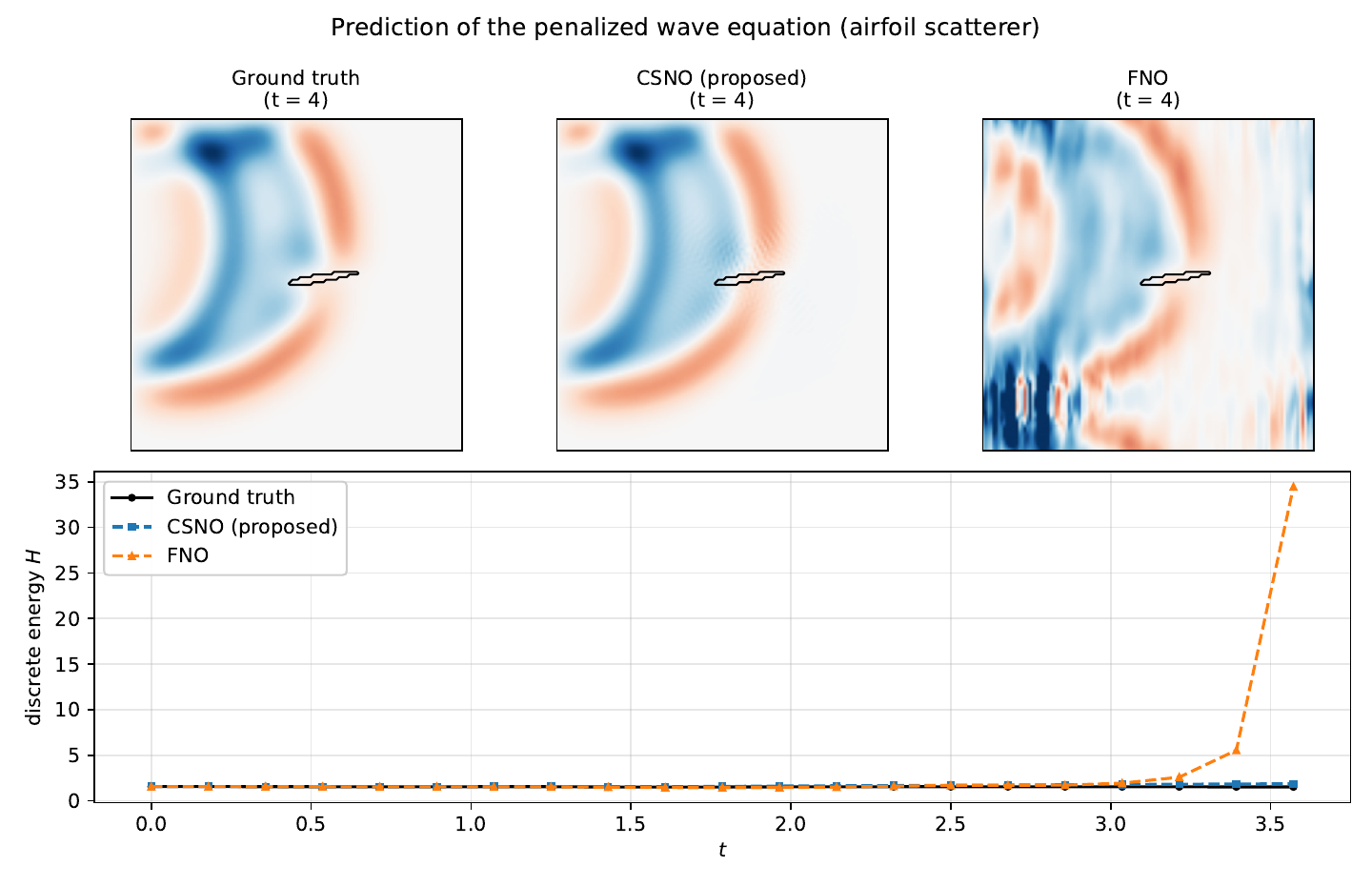}
\caption{Operator learning for the penalised wave equation with an
airfoil-shaped scatterer.  Top: ground truth, CSNO
(proposed), and FNO predictions of the wave field.  Bottom: discrete
energy $H$ versus time.  The FNO energy explodes unnaturally, while the proposed conformal symplectic neural operator tracks the ground-truth
energy closely.}
\label{fig:operator}
\end{figure}

\subsection{Maxwell's equations: conducting obstacles in 2D and 3D}
\label{subsec:exp-maxwell}

Our final experiment are the conductivity-penalised Maxwell
system \eqref{eq:maxwell-brinkman} integrated with the conformal
split-Yee scheme. Here we apply the exact dissipative substep
$\bE\mapsto e^{-\chi\Delta t/(2\eta)}\bE$ of
\eqref{eq:dissipation-semigroup} around a time-symmetric Yee
leapfrog, which is a composition of exact Hamiltonian shears in the
$\bK$-conjugate pair $(\bH,\bE)$ and hence multi-symplectic.

\paragraph{TM mode: validation and cross-check.}
\Cref{fig:maxwell-tm} shows TM scattering of an $E_z$ pulse by a
circular conductor (we apply a $256^2$ grid on $[-4,4]^2$ with PEC box
boundaries and penalisation parameter $\eta=0.02$). The scattering pattern reproduces \cref{fig:wave2d}, as you would expect given Remark \ref{rem:tm}. As we have seen throughout, the field expulsion from the conductor improves as $\eta\to0$, evidenced by the maximal
$|E_z|$ inside the conductor at $T=4$ decreasing from
$3.3\times10^{-2}$ ($\eta=0.1$) through $1.1\times10^{-2}$
($\eta=0.02$) to $4.1\times10^{-3}$ ($\eta=0.005$), the discrete
manifestation of the perfect-conductor limit of
\cref{subsec:brinkman}.  The TM setting also affords a quantitative
consistency cross-check of Remark \ref{rem:tm} between two different
structure-preserving realisations of the same continuum model. In the
fluid region the two-step recurrence for $E_z$ induced by the Yee
scheme coincides exactly with the Verlet recurrence of the wave
solver of \cref{subsec:exp-2d}, so the two split schemes differ only
in how the damping enters inside the solid ($E_z$ itself for
Maxwell versus $u_t$ for the wave formulation), an
$O(\Delta t/\eta)$ splitting-arrangement difference concentrated in
the penalisation boundary layer.  Consistently, the fluid-region
relative $L^2$ difference between the two computed solutions at
$T=4$ decreases essentially linearly in $\Delta t$ under refinement:
$0.154$, $0.077$, $0.051$ on $128^2$, $256^2$, $384^2$ grids
(CFL-proportional time steps, $\eta=0.02$), confirming that both
schemes converge to the same continuum solution.

\begin{figure}[t]
\centering
\includegraphics[width=\linewidth]{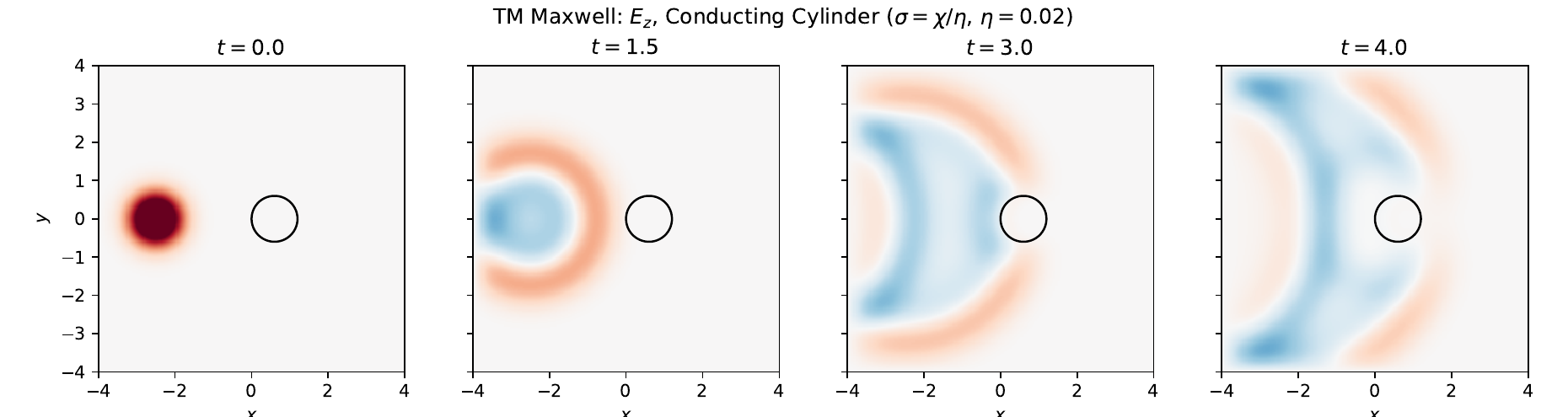}
\caption{TM Maxwell scattering of an $E_z$ pulse by a circular
conductor modeled by the conductivity penalisation
($\sigma=\chi/\eta$, $\eta=0.02$), computed with the conformal
split-Yee scheme.  By \cref{rem:tm} the configuration is equivalent
to \cref{fig:wave2d}, and the scattering pattern coincides.}
\label{fig:maxwell-tm}
\end{figure}

\paragraph{Fully three-dimensional scattering by a conducting
sphere.}
\Cref{fig:maxwell3d} shows the mid-plane $E_z$ field of a Gaussian
pulse scattering off a conducting sphere on a periodic Yee lattice
($64^3$, box $[-2,2]^3$, $\eta=0.02$), with the indicator function
evaluated at the Yee location of each electric-field component; the
maximal $|E_z|$ inside the conductor at $T=2.4$ is
$3.2\times10^{-2}$.  
\Cref{fig:maxwell3d-energy} reports the electromagnetic energy on a $48^3$ lattice up to $T=3$ against the exact Poynting budget
$H(0)-\int_0^t\!\int(\chi/\eta)|\bE|^2$ of \eqref{eq:poynting}. The
split-Yee scheme tracks this budget throughout, with a residual of $1.0\times10^{-2}$ at the level of the time-quadrature
while explicit Euler exhibits catastrophic energy growth.  Similar to \cref{subsec:exp-2d}, only the split scheme remains stable uniformly in $\eta$.

\begin{figure}[t]
\centering
\includegraphics[width=\linewidth]{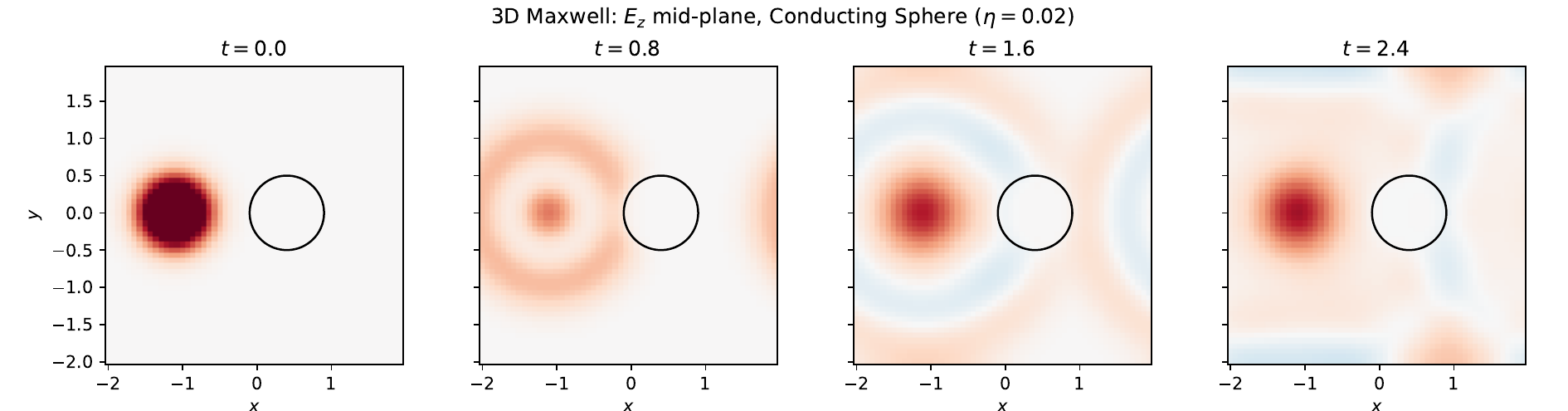}
\caption{Fully three-dimensional Maxwell scattering by a conducting
sphere via the conductivity penalisation ($\eta=0.02$, periodic Yee
lattice): mid-plane slices of $E_z$.  The field is expelled from the
conductor on the time scale $\eta$.}
\label{fig:maxwell3d}
\end{figure}

\begin{figure}[t]
\centering
\includegraphics[width=0.72\linewidth]{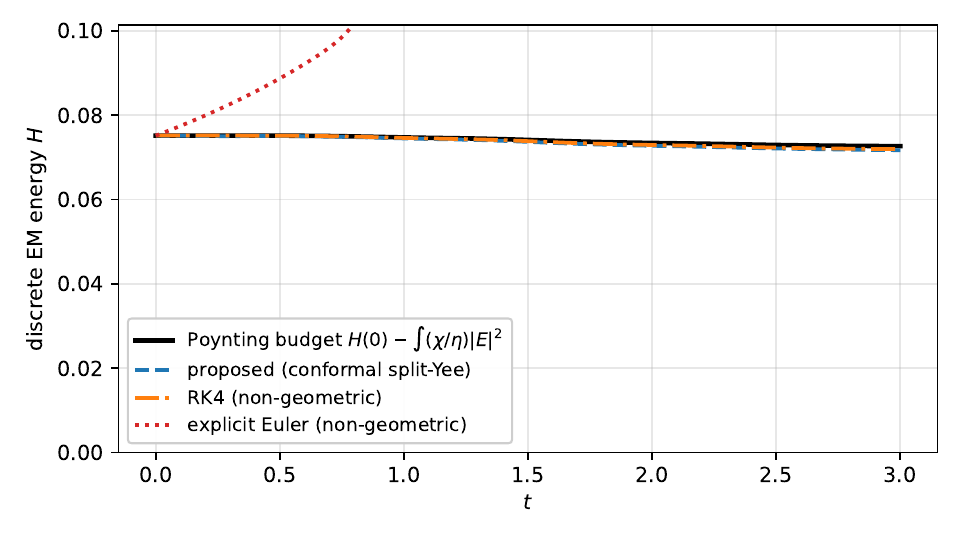}
\caption{Discrete electromagnetic energy for the 3D conducting-sphere
problem: exact Poynting budget $H(0)-\int(\chi/\eta)|\bE|^2$ (black),
the proposed conformal split-Yee scheme, RK4, and explicit Euler on
the same lattice and time step.  The proposed scheme tracks the
budget while explicit Euler exhibits catastrophic energy growth (leaving
the frame).}
\label{fig:maxwell3d-energy}
\end{figure}

\paragraph{Operator learning for three-dimensional Maxwell.}
Finally, we demonstrate the electromagnetic instantiation of the CSNO.  The dataset consists of 3D trajectories of \eqref{eq:maxwell-brinkman} past randomly sampled conducting obstacles past one or two randomly sampled spheres
excited by random Gaussian pulses. Once again these are generated with the conformal split-Yee integrator so that the data satisfy the discrete conformal and Poynting laws by construction. The CSNO takes the form \eqref{eq:csno} with the exact per-component
damping acting on $\bE$ and the learnable multi-symplectic operator
$\So_\theta$ acting on the $(\bH,\bE)$ conjugate pair. A fixed
Yee leapfrog of the vacuum Maxwell Hamiltonian is applied as the
leading layer, followed by kick/drift shears
\eqref{eq:sympnet-layers} whose three-dimensional convolutional
potential densities are conditioned on $\chi$ and initialized near
the identity.  The baseline is a three-dimensional FNO mapping
$(\bE,\bH,\chi)$ to the next state, trained on the same one-step data
with the same loss.  
Both models are trained for $150$ epochs of Adam (learning rate $10^{-3}$, batch size $8$) on a single NVIDIA L4 GPU. The CSNO uses two
kick/drift layer pairs of width $12$ ($2.5\times10^4$ parameters), and the FNO baseline uses three spectral layers with six Fourier modes per direction and width $12$ ($3.7\times10^5$ parameters, counting complex spectral weights as single parameters).

\Cref{fig:maxwell-csno} shows a rollout over the full prediction
window on a held-out geometry together with the electromagnetic
energy $H(t)$.  Over the test set, the relative $L^2$ error of the
electric field at the final frame is $0.073\pm0.017$ for the CSNO, which is over an order of magnitude smaller than the $0.988\pm0.015$ attained by the FNO baseline. The relative deviation of the final
energy from the ground truth is similar at $0.5\%$ for the CSNO versus $56\%$ for the FNO. The FNO prediction loses more than half of the electromagnetic energy through
spurious dissipation within a few characteristic times, while the
CSNO energy is visually indistinguishable from the ground truth
throughout, its dissipation confined to the conductors at the exact
rate prescribed by the penalisation
(\cref{prop:csno-law}).  Together with \cref{subsec:exp-op}, this
confirms that the conformal architecture transfers across the
penalised class as predicted by the theory.

\begin{figure}[t]
\centering
\includegraphics[width=0.92\linewidth]{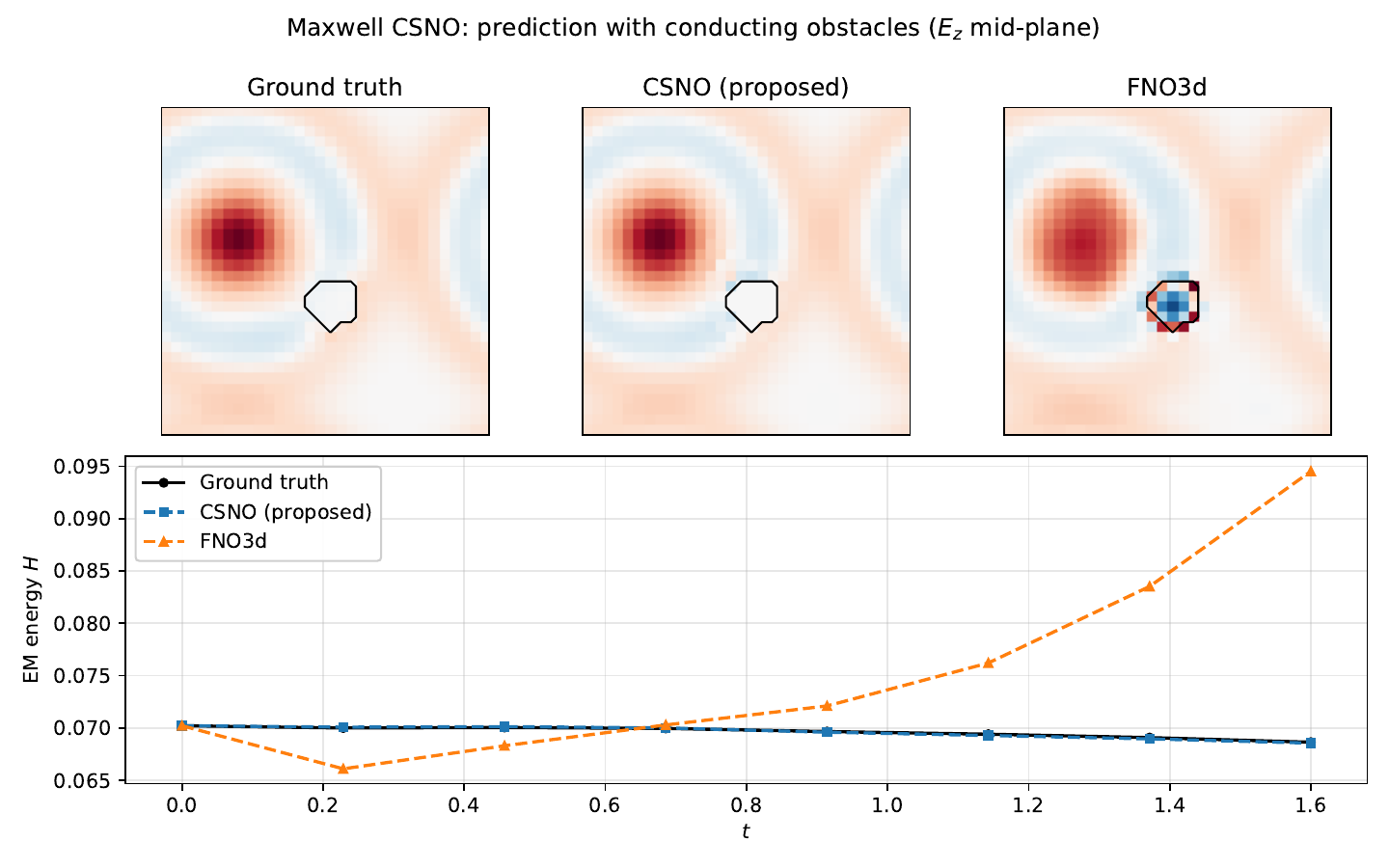}
\caption{Operator learning for the conductivity-penalised 3D Maxwell
equations with randomly sampled conducting obstacles.  Top:
mid-plane $E_z$ at the final frame of a rollout on a held-out
geometry. Left to right these show ground truth, electromagnetic CSNO (proposed), and FNO3d
baseline.  Bottom: electromagnetic energy $H(t)$ along the rollout.
The FNO energy decays spuriously to less than half of the true
value, while the CSNO tracks the ground-truth energy to within
$0.5\%$ at the final time.}
\label{fig:maxwell-csno}
\end{figure}

\section{Conclusion and future work}
\label{sec:conclusion}

We have shown that the Brinkman-type penalisations of multi-symplectic
Hamiltonian PDEs originally introduced as a numerical convenience, actually result in an exact geometric structure. Specifically, such systems have multi-conformal symplectic structure with spatially localised damping rate $\chi(\bx)/\eta$ whenever the penalised components are paired with unpenalised ones by the symplectic structure matrix. This realisation allows us to leverage structure preserving numerical methods to perform integration and operator learning on arbitrary geometries without specialised meshing or retraining, whilst attaining superior accuracy and conservation properties compared to standard approaches. We outline an algebraic condition expressed by the anticommutation relation
$\bK=\bP\bK+\bK\bP$, and show that this is satisfied by the wave equation with the classical Brinkman friction and by Maxwell's equations penalised by an artificial Ohmic conductivity
(\cref{thm:general,cor:wave,cor:maxwell}).  To establish this structure we derived an explicit quadratic modification
of the Hamiltonian density $\frac{\chi}{4\eta}\bz^{\mathsf T}(\bI-2\bP)\bK\bz$ that is supported on the solid region.  This observation allows us to view Brinkman penalised systems as members of a well-understood geometric class, and immediately suggests numerical methods such as Strang splitting 
which yield second-order schemes satisfying a discrete
conformal multi-symplectic conservation law and reproducing the correct
local energy budget (\cref{prop:discrete-law}).  The same composition,
with the conservative factor made learnable, defines conformal
symplectic neural operators that satisfy the conformal conservation
law by construction for arbitrary geometries supplied through the
indicator function (\cref{prop:csno-law}), uniformly across the
penalised class. Empirically we demonstrated that these methods avoid the unphysical energy drift exhibited by unconstrained neural operators, in both the acoustic and the fully three-dimensional electromagnetic settings.

Several directions remain open.  On the theoretical side, a
detailed treatment of nonlinear Hamiltonian PDEs in the penalised
class remains to be explored. This could, for example, include nonlinear wave and Schr\"odinger-type equations with
partially damped penalisations, together with convergence rates as
$\eta\to0$ that respect the conformal structure. Alternatively an investigation of the variational and geometric (Lagrangian, multi-Dirac) structures underlying multi-conformal symplectic PDEs could be explored. Another slightly tangential direction could consider the type of damping arising in \eqref{eq:naive-reduced} as an alternative to Brinkman style penalisation and investigate its intrinsically conformal structural analytically and numerically. If such a method is convergent it could then apply to every multi-symplectic PDE without a compatibility condition or modification of the Hamiltonian density.  On the computational side, the Maxwell experiments of \cref{subsec:exp-maxwell} validate the conformal split-Yee scheme on conducting obstacles in two and three dimensions and demonstrate the electromagnetic CSNO on families of random conductor geometries. Large-scale scattering by conductors of complex shape, and the coupling of the electromagnetic and acoustic settings within a single conformal architecture, are natural next targets.  On the learning side, an approximation theory for conformal symplectic neural operators, the treatment of moving and deforming geometries through time-dependent $\chi$, and the extension to dissipative fluid models could also be pursued.

\section*{Acknowledgments}
This work was supported by the JST CREST program (JPMJCR24Q5) on operator learning based on geometric classical field theory and infinite-dimensional data science, the JST ASPIRE program (JPMJAP2329) on deep scientific computing, the JST Moonshot program (JPMJMS24A3), the Maths4DL EPSRC Programme on the Mathematics of Deep Learning (EP/V026259/1), the UKRI Horizon Europe Guarantee scheme REMODEL (EP/Y032144/1), and the UKRI Strength in Places Fund MyWorld Project (SIPF00006/1).


\end{document}